\documentclass[12pt,english]{amsart}
\usepackage[T1]{fontenc}
\usepackage[latin9]{inputenc}
\usepackage{mathrsfs}
\usepackage{amstext}
\usepackage{amsthm}
\usepackage{amssymb}

\makeatletter
\numberwithin{equation}{section}
\numberwithin{figure}{section}
\theoremstyle{plain}
\newtheorem{thm}{\protect\theoremname}
\theoremstyle{plain}
\newtheorem{prop}[thm]{\protect\propositionname}
\theoremstyle{remark}
\newtheorem{rem}[thm]{\protect\remarkname}
\theoremstyle{definition}
\newtheorem{example}[thm]{\protect\examplename}
\theoremstyle{plain}
\newtheorem{cor}[thm]{\protect\corollaryname}
\theoremstyle{plain}
\newtheorem*{thm*}{\protect\theoremname}

\usepackage{hyperref}

\AtBeginDocument{

}

\makeatother

\usepackage{babel}
\usepackage[style=numeric,backend=bibtex]{biblatex}
\providecommand{\corollaryname}{Corollary}
\providecommand{\examplename}{Example}
\providecommand{\propositionname}{Proposition}
\providecommand{\remarkname}{Remark}
\providecommand{\theoremname}{Theorem}

\begin{document}
\global\long\def\dT{\widehat{T}}%

\global\long\def\dG{\widehat{G}}%

\global\long\def\dB{\widehat{B}}%

\global\long\def\dg{\widehat{\mathfrak{g}}}%

\global\long\def\db{\widehat{\mathfrak{b}}}%

\global\long\def\slC{SL(2,\mathbb{C})}%

\global\long\def\da{\check{\mathfrak{a}}}%

\global\long\def\dS{\widehat{S}}%

\global\long\def\wA{\widehat{A}}%

\title{A simple construction of the automorphic residual spectrum}
\author{Devadatta G. Hegde}
\address{Tata Institute of Fundamental Research, Mumbai.}
\email{dghegde@math.tifr.res.in}
\begin{abstract}
We consider the spherical Borel Eisenstein series induced from the
trivial representation for a split semisimple linear algebraic group
over a number field. We prove that its regularization at the special
point corresponding to half the weighted marking of a distinguished
coadjoint nilpotent orbit in the Langlands dual Lie algebra is nonzero
and square-integrable. Our proof follows the philosophy of Kazhdan
and Okounkov. We give a geometric interpretation of Langlands' square-integrability
criterion in this setting and, using the equivariant integration formula,
prove that the regularization satisfies this criterion. As an immediate
consequence, we obtain a simple and uniform proof of Arthur's unitarity
conjecture, without case-by-case analysis or machine computation. 
\end{abstract}

\maketitle
\tableofcontents{}

\section{Introduction}

Arthur's unitarity conjecture is widely recognized as important because
it bears on the problem of determining the unitary dual of a Chevalley
group over a local field of characteristic zero. The present investigation,
however, grew out of an attempt to understand and simplify Langlands'
construction of the residual automorphic spectrum via \textquotedbl iterated
residues.\textquotedbl{}

In the author's doctoral thesis \cite{Dattu-thesis}, a simple formula
was given for the poles of the spherical Eisenstein series induced
from the trivial representation on Levi subgroups of Chevalley groups;
the proof was simple and the result structural. The proof used \emph{Langlands'
residue formula }expressing non-cuspidal Eisenstein series as a ``residue''
of a cuspidal Eisenstein series in this case. This proof extended
to the general case \emph{only if }Langlands' construction of the
residual automorphic spectrum using a ``sophisticated residue scheme''
(\cite{Arthur_LanglandsWork}, section 2), in general, was a simple
regularization of the corresponding cuspidal Eisenstein series. The
present paper establishes this for the spherical Borel Eisenstein
series induced from the trivial representation --- a first test of
that broader claim.

As an immediate corollary of the main theorem \ref{thm:main}, we
obtain a \emph{uniform conceptual proof }of Arthur's unitarity conjecture,
the last cases of which was proved in 2013 using computers by Miller
\cite{Miller_Annals}. To elaborate, Arthur\textquoteright s conjecture
is that the spherical constituent of an unramified principal series
of a Chevalley group over any local field of characteristic zero is
unitarizable if its Langlands parameter coincides with half the weighted
marking of a coadjoint nilpotent orbit of the Langlands dual Lie algebra.
The conjecture was motivated by Langlands' construction of an automorphic
form attached to the subregular orbit of $G_{2}$ (\cite{Langlands-Spectral},
appendix III). Moeglin \cite{Moeglin_Compositio} proved the conjecture
for classical groups. For the remaining exceptional groups, Miller
verified the conjecture using machine computation \cite{Miller_Annals}.
In the absense of a uniform method, the standard approach seems to
be to use similar machine computations. For recent work in this direction,
see Halawi-Segal \cite{Segal-halawi} and Hundley-Miller \cite{Hundley-Miller}. 

\emph{If} Langlands' construction of the residual automorphic spectrum
is indeed a simple regularization in general, this would be an enormous
simplification of a foundational body of work that has seemed ``impenetrable''
(\cite{Langlands-Spectral}, preface) for at least two generations.
Part of the difficulty lies in the fact that Langlands' residue calculus
involves several non-canonical choices, even in the definition of
the iterated residue itself, and it is only after the fact that these
choices are shown not to affect the outcome. Readers curious about
the subtlety involved may find it instructive to work through remark
V.1.4 of the standard reference \cite{MoeglinWaldspurger-book}. Our
construction, by contrast, involves no such choices.

A recent breakthrough is the recognition by Kazhdan and Okounkov \cite{Kazhdan-Okounkov-Paper}
that the miraculous cancellations observed in Langlands' residue calculus
(\cite{Langlands-Spectral}, appendix III) are best explained through
equivariant localization applied to the geometry of homogeneous varieties
of the dual group. This is the main tool in our proof, though our
application of the equivariant integration formula, unlike theirs,
is standard and straightforward. Familiarity with their paper is not
required to read what follows; we hope our application conveys the
strength and utility of their idea in a small way and encourages wider
engagement with their approach to spectral decomposition of automorphic
forms. 

\subsection{\label{subsec:Notation}Notation}

We use the notation introduced here throughout the paper.
\begin{itemize}
\item $G$: connected split semisimple linear algebraic group over a number
field $F$. 
\item $B=T\rtimes U$: fixed Borel subgroup $B$ of $G$, with $T$ a maximal
torus, and $U$ the unipotent radical of $B$.
\item Set
\[
\mathfrak{a}=X_{*}(T)\otimes\mathbb{R},\quad\check{\mathfrak{a}}=X^{*}(T)\otimes\mathbb{R},\quad\check{\mathfrak{a}}_{\mathbb{C}}=\check{\mathfrak{a}}\otimes\mathbb{C}
\]
where $X_{*}(T)$ and $X^{*}(T)$ denote co-character and characters
of $T$. 
\item $\Delta$: the set of simple roots of $G$ with respect to $B$. 
\item $\left\{ \varpi_{\alpha}\right\} _{\alpha\in\Delta}\subset\da$ be
the fundamental weights, dual to the simple coroots. 
\item $W$: the Weyl group of $G$. 
\item $\rho\in\da$: half-sum of positive roots. 
\item Weight cone: 
\[
\mathscr{C}_{\text{wt}}:=\left\{ \Lambda\in\da:\langle\Lambda,\alpha^{\vee}\rangle>0\ \forall\alpha\in\Delta\right\} 
\]
\[
=\left\{ \Lambda=\sum_{\alpha\in\Delta}t_{\alpha}\varpi_{\alpha}:t_{\alpha}>0\ \forall\alpha\in\Delta\right\} \subset\da
\]
\item Root cone: 
\[
\mathscr{C}_{\text{rt}}:=\left\{ \Lambda=\sum a_{\alpha}\alpha:a_{\alpha}>0\ \ \forall\alpha\in\Delta\right\} \subset\da
\]
\item $\dG$: Langlands dual group of $G$
\item $\dB\supset\dT$: Borel subgroup $\dB$ of $\dG$ with maximal torus
$\dT$, 
\[
X^{*}(\dT)=X_{*}(T)\qquad X_{*}(\dT)=X^{*}(T).
\]
\item $[G]:=G(F)\backslash G(\mathbb{A})$. Fix a nonzero invariant measure
on $[G]$ throughout.
\item $\mathcal{A}[G]$: the space of automorphic forms on $[G]$.
\item $L^{2}[G]$: the Hilbert space of square-integrable functions on $[G]$. 
\item $H=H_{B}:G(\mathbb{A})\to\mathfrak{a}$, the Harish-Chandra map. 
\end{itemize}

\subsection{Statement of the main theorem}

The \emph{unramified Borel Eisenstein series} is defined initially
as an absolutely convergent sum 
\[
E(\Lambda;g)=\sum_{\delta\in B(F)\backslash G(F)}e^{\langle\rho+\Lambda,H(\delta g)\rangle}\quad\Lambda\in\da_{\mathbb{C}},\ g\in G(\mathbb{A})
\]
when $\Re(\Lambda)\in\rho+\mathscr{C}_{\text{wt}}$. It admits a meromorphic
continuation to all of $\da_{\mathbb{C}}$ \cite{Lapid_Bernstein}.
We write 
\[
E(\Lambda):=E(\Lambda;\bullet)\in\mathcal{A}[G]
\]
to lighten the notation. 

The poles and zeros of $E(\Lambda)$ relevant for spectral decomposition
is captured by 
\[
D(\Lambda):=\prod_{\gamma>0}\frac{\langle\Lambda,\gamma^{\vee}\rangle}{\langle\Lambda,\gamma^{\vee}\rangle-1}
\]
Let
\[
E^{*}(\Lambda):=D(\Lambda)^{-1}E(\Lambda)
\]
It is known (Jacquet \cite{Jacquet_residual_GL_n}, Hegde \cite{Dattu-thesis}
p.31) that $E^{*}$ is never zero, and that it is holomorphic on the
tube domain over the closed dominant chamber $\overline{\mathscr{C}_{\text{wt}}}$. 

Let 
\[
\kappa:\slC\to\dG
\]
be a distinguished homomorphism, meaning $\kappa(SL(2,\mathbb{C}))$
is not contained in any \emph{proper} Levi subgroup of $\widehat{G}$.
Let $\mathbf{G}_{m}$ be the diagonal torus in $\slC$. We assume
that
\[
\kappa\big\vert_{\mathbf{G}_{m}}=:h_{\kappa}\in X_{*}(\dT)\simeq X^{*}(T)
\]
is dominant, that is, $h_{\kappa}\in\overline{\mathscr{C}_{\text{wt}}}$.
Let 
\[
\rho_{\kappa}:=h_{\kappa}/2
\]
When $\kappa$ is the \emph{principal} $\slC$ homomorphism, $\rho_{\kappa}=\rho$
is the half-sum of positive roots. 
\begin{thm}
\label{thm:main}The automorphic form $E^{*}(\rho_{\kappa})$ is non-zero
and square-integrable, that is, 
\[
0\neq E^{*}(\rho_{\kappa})\in L^{2}[G].
\]
\end{thm}

The formulation in Theorem \ref{thm:main} is due to the author. The
automorphic form $E^{*}(\rho_{\kappa})$ may be obtained in multiple
ways as an ``iterated residue'' of $E(\Lambda)$ and hence a version
of this theorem is in the union of \cite{Langlands-Spectral,Miller_Annals,Moeglin_Compositio}.
Langlands proved the result for $G_{2}$ after a long calculation
(\cite{Langlands-Spectral}, appendix III). Moeglin \cite{Moeglin_Compositio}
proved the result for classical groups. Miller \cite{Miller_Annals}
proved the theorem for the remaining exceptional groups, essentially
by using a computer to verify that the constant terms of $E^{*}(\rho_{\kappa})$
along the Borel satisfy Langlands' square-integrability criterion.
Interpreting Miller's formulation together with the author's doctoral
thesis suggested the formulation in Theorem \ref{thm:main}. 

\subsection{Arthur's unitarity conjecture}

Let $\pi_{\kappa}\subset L^{2}[G]$ be the representation generated
by $E^{*}(\rho_{\kappa})$. This representation is irreducible, and
hence factors as a restricted tensor product
\[
\pi_{\kappa}=\bigotimes_{v\le\infty}\pi_{\kappa,v}
\]
The representation $\pi_{\kappa,v}$ is the spherical constituent
of the principal series representation 
\[
\text{Ind}_{B(F_{v})}^{G(F_{v})}\left(e^{\langle\rho_{\kappa}+\rho,H_{v}(\cdot)\rangle}\right),
\]
where $H_{v}$ is the local Harish-Chandra map. The main Theorem \ref{thm:main}
shows that $\pi_{\kappa,v}$ is unitary, which is Arthur's unitarity
conjecture \cite{Arthur_conjecture} (see Miller \cite{Miller_Annals}
for more details). 

\subsection*{Acknowledgments}

I would like to thank Prof. Mohan Swaminathan for sharing his notes
on equivariant cohomology, which were of great help to me. I am also
grateful to Professors Arvind Nair, Dipendra Prasad, and Sandeep Varma for their encouragement.

\section{Langlands' square-integrability criterion applied to $E^{*}(\rho_{\kappa})$}

The constant term of an automorphic form $f\in\mathcal{A}[G]$ along
$B=U\rtimes T$ is 
\[
\left(Cf\right)(g):=\int_{U(F)\backslash U(\mathbb{A})}f(ug)du
\]
To lighten the notation, we drop the parantheses and write $Cf$. 

We have 
\[
CE(\Lambda;g)=\sum_{w\in W}M(w,\Lambda)e^{\langle\rho+w\Lambda,H(g)\rangle}
\]
where 
\[
M(w,\Lambda)=\prod_{\gamma>0:w\gamma<0}\frac{\xi(\langle\Lambda,\gamma^{\vee}\rangle)}{\xi(\langle\Lambda,\gamma^{\vee}\rangle+1)}
\]
Here, $\xi$ is the completed Dedekind zeta function of the number
field $F$. It is holomorphic on $\mathbb{C}-\{0,1\}$, with simple
poles at $0$ and $1$. The constant term 

\[
CE^{*}(\Lambda;g)=\sum_{w\in W}M^{*}(w,\Lambda)e^{\langle\rho+w\Lambda,H(g)\rangle},
\]
where 
\[
M^{*}(w,\Lambda):=D(\Lambda)^{-1}M(w,\Lambda)
\]

\subsection{Expression for $CE^{*}(\rho_{\kappa})$}

Let 
\[
\wA:=\kappa\left(\mathbf{G}_{m}\right)\subset\dT
\]
and 
\[
\widehat{L}_{\kappa}:=Z_{\dG}(\wA)
\]
be its centralizer in $\dG$. This is a Levi subgroup of $\dG$; let
$W_{\kappa}$ be its Weyl group. Then
\[
W_{\kappa}=\left\{ w\in W:w\cdot\rho_{\kappa}=\rho_{\kappa}\right\} .
\]

Since the several of the exponents appearing in $CE^{*}(\rho_{\kappa})$
collide, we group the terms into cosets of $W_{\kappa}$. 
\begin{prop}
For each $\omega\in W/W_{\kappa}$, there exists a polynomial $P_{\omega}\in\mathbb{C}[\mathfrak{a}]$
such that 
\[
CE^{*}(\rho_{\kappa};g)=\sum_{\omega\in W/W_{\kappa}}P_{\omega}(H(g))e^{\langle\rho+\omega\rho_{\kappa},H(g)\rangle}
\]
\end{prop}

\begin{proof}
Write 
\[
\Lambda=\rho_{\kappa}+\nu
\]
Then 
\[
CE^{*}(\rho_{\kappa}+\nu;g)=\sum_{w\in W}M^{*}(w,\rho_{\kappa}+\nu)e^{\langle\rho+w(\rho_{\kappa}+\nu),H(g)\rangle}
\]
\[
=\sum_{\omega\in W/W_{\kappa}}\left(\sum_{\sigma\in W_{\kappa}}\left(M^{*}(\omega\sigma,\rho_{\kappa}+\nu)e^{\langle\omega\sigma\nu,H(g)\rangle}\right)\right)e^{\langle\rho+\omega\rho_{\kappa},H(g)\rangle}
\]
Since $\lim_{\Lambda\to\rho_{\kappa}}E^{*}(\Lambda)$ exists and $\left\{ \omega\rho_{\kappa}:\omega\in W/W_{\kappa}\right\} $
are distinct elements of $\da$, the following limit exists
\[
P_{\omega}(H(g)):=\lim_{\nu\to0}\left(\sum_{\sigma\in W_{\kappa}}M^{*}(\omega\sigma,\rho_{\kappa}+\nu)e^{\langle\omega\sigma(\nu),H(g)\rangle}\right)
\]
This limit is a finite linear combination of derivatives of $e^{\langle\cdot,H\rangle}$
at $0$, which produces a polynomial in $H$. 
\end{proof}
\begin{rem}
Explicit expressions for $P_{\omega}(H)$ are very complicated in
general. The equivariant integration formula will enable their examination
by providing an integral representation for $P_{\omega}(H)$. 
\end{rem}

The following proposition is the famous Langlands' square-integrability
criterion applied to $E^{*}(\rho_{\kappa})$.
\begin{prop}
\label{prop:Langlands-criterion}If, for every $\omega\in W/W_{\kappa}$,
\[
\omega\cdot\rho_{\kappa}\notin-\mathscr{C}_{\text{rt}}\ \implies\ P_{\omega}(H)=0,
\]
then
\[
E^{*}(\rho_{\kappa})\in L^{2}[G].
\]
\end{prop}

\section{Geometric recasting of the square-integrability criterion}

In this section, we recast the condition 
\[
\omega\rho_{\kappa}\notin-\mathscr{C}_{\text{rt}},\qquad\omega\in W/W_{\kappa},
\]
in Proposition \ref{prop:Langlands-criterion}, geometrically, in
terms of the vanishing of the ordinary (not equivariant) Euler class
of certain a vector bundle $E_{\omega}$. 

Fix the triple $\{e,h,f\}\subset\dg$ via $d\kappa:\mathfrak{sl}(2,\mathbb{C})\to\dg$:
\[
e=d\kappa\begin{pmatrix}0 & 1\\
0 & 0
\end{pmatrix},\quad h=d\kappa\begin{pmatrix}1 & 0\\
0 & -1
\end{pmatrix},\quad f=d\kappa\begin{pmatrix}0 & 0\\
1 & 0
\end{pmatrix}
\]

\subsection{Torus action on $\widehat{G}/\widehat{B}$}

We recall some general facts. The flag variety of $\dG$ is
\[
\mathcal{B}:=\widehat{G}/\widehat{B}=\left\{ \text{Borel subalgebras of }\dg\right\} 
\]
For a torus $\widehat{S}\subset\widehat{T}$, the centralizer 
\[
L_{\widehat{S}}=Z_{\dG}\left(\widehat{S}\right)
\]
is a connected reductive group. If $W_{\dS}$ denotes its Weyl group,
then the fixed point set 
\[
\mathcal{B}^{\widehat{S}}=\bigsqcup_{u\in W/W_{\widehat{S}}}\mathcal{C}_{u},
\]
where each connected component $\mathcal{C}_{u}$ of $\mathcal{B}^{\dS}$
is isomorphic to the flag variety of $L_{\widehat{S}}$. 

In the extreme case of $\widehat{S}=\dT$, we have 
\[
L_{\widehat{T}}=\dT\quad\text{and}\quad\mathcal{B}^{\dT}=\bigsqcup_{w\in W}p_{w},\ \ p_{w}=w^{-1}\widehat{\mathfrak{b}}w
\]

\subsection{Geometric reformulation of Langlands' criterion}

Let $\wA=\kappa(\mathbf{G}_{m})$ be as before. Then 
\[
\mathcal{B}^{\wA}=\mathcal{B}^{h}=\left\{ \widehat{\mathfrak{q}}\in\mathcal{B}:h\in\widehat{\mathfrak{q}}\right\} =\bigsqcup_{\omega\in W/W_{\kappa}}\mathcal{C}_{\omega}
\]
Let 
\[
\mathcal{B}^{f}:=\left\{ \widehat{\mathfrak{q}}\in\mathcal{B}:f\in\widehat{\mathfrak{q}}\right\} 
\]

\begin{prop}
With the previous notation, 

(a) If $\kappa$ is not assumed to be distinguished, then 
\[
\mathcal{B}^{f}\cap\mathcal{C}_{\omega}\neq\emptyset\quad\implies\quad\omega\cdot\rho_{\kappa}\in-\overline{\mathscr{C}_{\text{rt}}}
\]

(b) If $\kappa$ is distinguished, then 
\[
\mathcal{B}^{f}\cap\mathcal{C}_{\omega}\neq\emptyset\quad\implies\quad\omega\cdot\rho_{\kappa}\in-\mathscr{C}_{\text{rt}}
\]
\end{prop}

\begin{proof}
Let $\widehat{\mathfrak{q}}\in\mathcal{B}^{f}\cap\mathcal{C}_{\omega}$,
and write
\[
\omega\cdot\rho_{\kappa}=\sum_{\alpha\in\Delta}a_{\alpha}\alpha,\quad a_{\alpha}\in\mathbb{R}\ \forall\alpha\in\Delta
\]
Let $\varpi_{\alpha^{\vee}}$ be the fundamental co-weight associated
with $\alpha\in\Delta$, and let $V_{\alpha}$ be an irreducible representation
of $\dG$ with the line $\ell_{\alpha}\subset V_{\alpha}$, stabilized
by $\dB$, on which $\dT$ acts by 
\[
\chi_{\alpha^{\vee}}=N_{\alpha}\varpi_{\alpha^{\vee}}\in\mathfrak{a}\quad\text{for some }N_{\alpha}\in\mathbb{Z}_{\ge1}.
\]
The Borel subalgebra $\widehat{\mathfrak{q}}$ stabilizes a line $\ell_{\omega,\alpha}\subset V_{\alpha}$.
Then $h\in\widehat{\mathfrak{q}}$ acts on $\ell_{\omega,\alpha}$
by 
\[
h\cdot v=2\langle\omega\cdot\rho_{\kappa},\chi_{\alpha^{\vee}}\rangle v,\quad\forall v\in\ell_{\omega,\alpha}
\]
Since $f\in\widehat{\mathfrak{q}}$ is nilpotent, we have 
\[
f\cdot v=0,\quad\forall v\in\ell_{\omega,\alpha}.
\]
By $\mathfrak{sl}(2,\mathbb{C})$-representation theory, the $h$-eigenvalue
\[
2\langle\omega\cdot\rho_{\kappa},\chi_{\alpha^{\vee}}\rangle=2N_{\alpha}a_{\alpha}\le0
\]
This proves part (a). 

Assume that $\kappa$ is distinguished. If $a_{\alpha}=0$ for some
$\alpha\in\Delta$, then 
\[
f\cdot v=h\cdot v=e\cdot v=0,\quad\forall v\in\ell_{\omega,\alpha}
\]
Thus $\kappa(\slC)$ stabilizes the line $\ell_{\omega,\alpha}$,
and hence lies in a proper parabolic subgroup of $\dG$. Since $\kappa(\slC)$
is reductive, it lies in a proper Levi subgroup of $\dG$. Since the
latter is forbidden when $\kappa$ is distinguished, we have $a_{\alpha}<0$
for all $\alpha\in\Delta$. 
\end{proof}
Let 
\[
\mathcal{V}=T\mathcal{B}\quad\text{and}\quad\mathcal{V}_{\omega}=T\mathcal{B}\big\vert_{\mathcal{C}_{\omega}}
\]
Since $\wA\simeq\mathbb{C}^{\times}$ acts trivially on each component
$\mathcal{C}_{\omega}$, we have the $\wA$-weight decomposition of
the bundle 
\[
\mathcal{V}_{\omega}=\bigoplus_{k\in\mathbb{Z}}\mathcal{V}_{\omega}[k]
\]
Let 
\[
E_{\omega}:=\mathcal{V}_{\omega}[-2]
\]
The reader may compare the following proposition with Lemma 3.5 of
\cite{Kazhdan-Okounkov-Paper}.
\begin{thm}
\label{thm:euler-vanish}Let $\kappa$ be distinguished. For $\omega\in W/W_{\kappa}$,
we have 
\[
\omega\cdot\rho_{\kappa}\notin-\mathscr{C}_{\text{rt}}\quad\implies\quad e\left(E_{\omega}\right)=0,
\]
where $e(\cdot)$ denotes the ordinary Euler class of a vector bundle. 
\end{thm}

\begin{proof}
Consider the section 
\[
s_{f}:\mathcal{C}_{\omega}\to E_{\omega},\quad s_{f}(\widehat{\mathfrak{q}})=f\mod\widehat{\mathfrak{q}}
\]
The vanishing set $Z(s_{f})$ of section $s_{f}$ is 
\[
Z(s_{f})=\left\{ \widehat{\mathfrak{q}}\in\mathcal{C}_{\omega}:f\in\widehat{\mathfrak{q}}\right\} =\mathcal{B}^{f}\cap\mathcal{C}_{\omega}
\]
By the previous proposition, 
\[
\omega\cdot\rho_{\kappa}\notin-\mathscr{C}_{\text{rt}}\quad\implies\quad\mathcal{B}^{f}\cap\mathcal{C}_{\omega}=\emptyset=Z(s_{f}).
\]
Thus $s_{f}$ is a nowhere vanishing section of $E_{\omega}$. By
standard results in the geometry of characteristic classes, we have
\[
e\left(E_{\omega}\right)=0.
\]
\end{proof}

\section{Equivariant cohomology and the integration formula}

A suitable application of the \emph{equivariant integration formula}
gives
\[
P_{\omega}=\int_{\mathcal{C}_{\omega}}e\left(E_{\omega}\right)(\text{other classes}),\quad\forall\omega\in W/W_{\kappa}
\]
This immediately shows that $E^{*}(\rho_{\kappa})$ satisfies Langlands'
square-integrability criterion.

In this section, we recall the \emph{equivariant localization package}
we need. Two excellent sources are \cite{Fulton-Anderson} and \cite{Atiyah-Bott-Equivariant}.
For a topological space $X$, we write $H^{*}X$ or $H^{*}(X)$for
its singular cohomology with coefficients in $\mathbb{C}$.

\subsection{Definitions}

If $\dT\simeq\left(\mathbb{C}^{\times}\right)^{r}$, let 
\[
E\dT=\left(\mathbb{C}^{\infty}-\{0\}\right)^{r},\quad B\dT=E\dT/\dT=(\mathbb{CP}^{\infty})^{r}.
\]
There is a canonical isomorphism 
\[
R:=H^{*}(B\dT)\simeq\text{Sym}^{*}\left(X^{*}(\dT)\right)=\mathbb{C}[\da_{\mathbb{C}}]
\]
The group operation in $X^{*}(\dT)$ is written additively and characters
in cohomological degree $2$. A coroot $\gamma^{\vee}$ defines a
linear polynomial 
\[
\Lambda\mapsto\langle\Lambda,\gamma^{\vee}\rangle,\quad\Lambda\in\da_{\mathbb{C}}.
\]
Products of coroots are of higher degree. 

For a left $\dT$ space $Y$, its \emph{Borel construction} is 
\[
Y_{\dT}:=\frac{E\dT\times Y}{(e\cdot t,y)\sim(e,t\cdot y)},
\]
and the $\dT$-equivariant cohomology is
\[
H_{\dT}^{*}(Y):=H^{*}(Y_{\dT}).
\]
The projection $Y_{\dT}\to B\dT$ makes $H_{\dT}^{*}(Y)$ a module
over $R$. 
\begin{example}
(Important) If $V$ is a $\dT$-representation with weights $\chi_{1},\dots,\chi_{d}$,
then we can view it as the equivariant vector bundle over a point
$\mathbf{p}$ on which $\dT$ acts trivially. Then 
\[
H_{\dT}^{*}(\mathbf{p}):=H^{*}(E\dT\times_{\dT}\mathbf{p})=H^{*}(E\dT/\dT)=H^{*}(B\dT)
\]
The Euler class $e_{\dT}(V)$ of $V$ is 
\[
e_{\dT}(V)=\prod_{j=1}^{d}\chi_{j}\in H^{2d}(B\dT)
\]
At a fixed point for the $\dT$-action, the geometric language captures
the representation-theoretic information. 
\end{example}

Fix $\mu\in\da$. Let $\mathscr{O}_{\mu}$ be the local ring of holomorphic
germs at $\mu$, and let 
\[
\mathscr{M}_{\mu}=\text{Frac}(\mathscr{O}_{\mu})
\]
be its field of meromorphic germs. We have the inclusions
\[
R\to\mathscr{O}_{\mu}\to\mathscr{M}_{\mu}.
\]
For a flag variety $Y$, write
\[
H_{\dT,\text{an}}^{*}(Y)_{\mu}:=H_{\dT}^{*}(Y)\otimes_{R}\mathscr{O}_{\mu}
\]
\[
H_{\dT,\text{mer}}^{*}(Y)_{\mu}:=H_{\dT}^{*}(Y)\otimes_{R}\mathscr{M}_{\mu}
\]

\subsection{The localization package}

Assume that $Y$ is a flag variety with $\dT$-action and that the
fixed point set $Y^{\dT}$ is finite. The following is a standard
result (\cite{Fulton-Anderson}, Theorem 5.1.13). 
\begin{thm}
\label{thm:localization-package}Let 
\[
i:Y^{\dT}\to Y,\quad\pi:Y\to\text{pt}
\]
be the inclusion and projection maps respectively. 

(a) (Injectivity at the level of holomorphic germs) The induced map
\[
i^{*}:H_{\dT}^{*}(Y)\to H_{\dT}^{*}(Y^{\dT})
\]
is injective and remains injective after extension to holomorphic
germs:
\[
i^{*}:H_{\dT,\text{an}}^{*}(Y)_{\mu}\hookrightarrow\bigoplus_{p\in Y^{\dT}}\mathscr{O}_{\mu}
\]
The image of $i^{*}$ is given by an explicit GKM condition (\cite{Fulton-Anderson},
Corollary 7.4.3 and section 15.4). 

(b) (Isomorphism) After extension to meromorphic germs, 
\[
i^{*}:H_{\dT,\text{mer}}^{*}(Y)_{\mu}\simeq\bigoplus_{p\in Y^{\dT}}\mathscr{M}_{\mu}
\]
is an isomorphism. 

(c) (Integration formula) The equivariant pushforward extends by scalars
to
\[
\pi_{*}:H_{\dT,\text{mer}}^{*}(Y)_{\mu}\to\mathscr{M}_{\mu}
\]
and satisfies 
\[
\int_{Y}^{\dT}\eta:=\pi_{*}(\eta)=\sum_{p\in Y^{\dT}}\frac{\eta\vert_{p}}{e_{\dT}(T_{p}Y)}
\]
as an equality of meromorphic germs. 
\end{thm}

The following example illustrates how the integration formula will
be applied. 
\begin{example}
Let $V$ be a $\dT$-equivariant complex vector bundle of rank $m$
over $Y$ of dimension $n$ with finite fixed point set $Y^{\dT}$.
Let $\phi$ be a holomorphic function and let $y_{1}^{\dT},\dots,y_{m}^{\dT}$
be the equivariant Chern roots of $V$. We can define an equivariant
class 
\[
\Theta_{\phi}(V)=\prod_{j=1}^{m}\phi(y_{j}^{\dT})\ \in H_{\dT,\text{an}}^{*}(Y)_{\mu}
\]
Let the weights of $\dT$ action at the fiber $V_{p}$ over a fixed
point $p\in Y^{\dT}$ be $\beta_{p,1},\dots,\beta_{p,m}$. Let the
weights for $T_{p}Y$ be $\alpha_{p,1},\dots,\alpha_{p,n}$. Note
that for $p\in Y^{\dT}$, we have 
\[
\Theta_{\phi}(V)\big\vert_{p}=\left(\prod_{j=1}^{m}\phi(y_{j}^{\dT})\right)\big\vert_{p}=\prod_{j=1}^{m}\phi(y_{j}^{\dT}\big\vert_{p})=\prod_{j=1}^{m}\phi(\beta_{p,j})
\]
The equivariant integration formula says 
\[
\int_{Y}^{\dT}\Theta_{\phi}(V)=\sum_{p\in Y^{\dT}}\frac{\Theta_{\phi}(V)\big\vert_{p}}{e_{\dT}(T_{p}Y)}=\sum_{p\in Y^{\dT}}\frac{\prod_{j=1}^{m}\phi(\beta_{p,j})}{\prod_{i=1}^{n}\alpha_{p,i}}
\]
This is an equality of meromorphic germs at $\mu$. Even though the
individual terms $\frac{\prod_{j=1}^{m}\phi(\beta_{p,j})}{\prod_{i=1}^{n}\alpha_{p,i}}$
are only meromorphic at $\mu$, their sum is holomorphic at $\mu$.
This is analogous to the limit expression for $P_{\omega}(H)$. 

Later we use $\phi(s)=(s^{2}-1)\xi(s)$ which is holomorphic everywhere
except $0$. This requires some care. 
\end{example}

\section{Vanishing of the coefficients obstructing square-integrability}

To obtain an integral representation for $P_{\omega}(H)$, we need
a few auxiliary functions. 

\subsection{Auxiliary functions $\phi,j,J$}

The function
\[
c(s):=\frac{\xi(s)}{\xi(1+s)}
\]
appears in the constant term formula for the Eisenstein series $E(\Lambda)$.
Let
\[
\phi(s):=(s^{2}-1)\xi(s).
\]
The function $\phi$ has a simple pole at $0$ and a simple zero at
$-1$; otherwise it is holomorphic.

Define $j(s)$ to satisfy 
\[
j(s)\phi(-s)=\frac{s-1}{s}\quad\text{and}\quad j(s)\phi(s)=\frac{s-1}{s}c(s)
\]
Explicitly, after using $\xi(-s)=\xi(1+s)$, we have
\[
j(s)=\frac{1}{s(s+1)\xi(s+1)}.
\]
 The function $j$ is holomorphic and nonzero on $[0,\infty)$. 
\begin{prop}
The function
\[
J(\Lambda)=\prod_{\gamma>0}j(\langle\Lambda,\gamma^{\vee}\rangle)
\]
is holomorphic and nonzero near $\rho_{\kappa}$.
\end{prop}

\begin{proof}
Follows since $\rho_{\kappa}$ is dominant. 
\end{proof}

\subsection{Construction of an equivariant class}

For a $\dT$-equivariant vector bundle $V$ of rank $d$ with equivariant
Chern roots $y_{1}^{\dT},\dots,y_{d}^{\dT}$, define the meromorphic
class 
\[
\Theta_{\phi}(V)=\prod_{j=1}^{d}\phi(y_{j}^{\dT})
\]
when it is well-defined. With the convention that 
\[
p_{w}:=w^{-1}\db w\in\mathcal{B}\quad(w\in W),
\]
the tangent weights of $\mathcal{B}$ at $p_{w}$ are 
\[
\text{Wt}_{\dT}(T_{p_{w}}\mathcal{B})=\left\{ \gamma^{\vee}:\gamma>0,w\gamma<0\right\} \bigsqcup\left\{ -\gamma^{\vee}:\gamma>0,w\gamma>0\right\} 
\]
The above choice of $\phi,j,J$ was made so that 
\begin{prop}
We have
\[
J(\Lambda)\Theta_{\phi}(T\mathcal{B})\big\vert_{p_{w}}(\Lambda)=M^{*}(w,\Lambda)
\]
\end{prop}

\begin{proof}
This is a simple computation: 
\[
J(\Lambda)\Theta_{\phi}(T\mathcal{B})\big\vert_{p_{w}}(\Lambda)=
\]
\[
\prod_{\gamma>0}j(\langle\Lambda,\gamma^{\vee}\rangle)\cdot\prod_{\gamma>0:w\gamma<0}\phi(\langle\Lambda,\gamma^{\vee}\rangle)\prod_{\gamma>0:w\gamma>0}\phi(-\langle\Lambda,\gamma^{\vee}\rangle)
\]
\[
=D(\Lambda)^{-1}\prod_{\gamma>0:w\gamma<0}c(\langle\Lambda,\gamma^{\vee}\rangle)=M^{*}(w,\Lambda).
\]
\end{proof}

\subsection{The exponential class}

Fix $H\in\mathfrak{a}$. At a fixed point $p_{w}$ of $\mathcal{B}$,
prescribe the holomorphic germ 
\[
\varepsilon_{w,H}(\Lambda)=e^{\langle\rho+w\Lambda,H\rangle}.
\]

\begin{prop}
There is a global analytic class 
\[
\text{Exp}_{H}\in H_{\dT,\text{an}}^{*}\left(\mathcal{C}_{\omega}\right)_{\rho_{\kappa}}
\]
satisfying 
\[
\text{Exp}_{H}\big\vert_{p_{w}}(\Lambda)=\varepsilon_{w,H}(\Lambda)
\]
\end{prop}

\begin{proof}
This is a simple application of the GKM divisibility condition applied
to our situation (\cite{Fulton-Anderson}, Corollary 7.4.3 and Proposition
15.4.2). For a positive root $\gamma>0$, let $s_{\gamma}\in W$ be
the reflection defined by $\gamma$. In the present situation, the
GKM divisibility condition is 
\[
\exp\langle\rho+w\Lambda,H\rangle-\exp\langle\rho+ws_{\gamma}\Lambda,H\rangle\in\langle\Lambda,\gamma^{\vee}\rangle\cdot\mathscr{O}_{\rho_{\kappa}}
\]
for every positive root $\gamma>0$ and every $w\in W$. This easily
follows from 
\[
\Lambda-s_{\gamma}\Lambda=\langle\Lambda,\gamma^{\vee}\rangle\gamma,
\]
which implies
\[
\langle w\Lambda,H\rangle-\langle ws_{\gamma}\Lambda,H\rangle=\langle\Lambda,\gamma^{\vee}\rangle\langle w\gamma,H\rangle
\]
\end{proof}
Previous computations immediately yield
\begin{cor}
For every $w\in W$ and $H\in\mathfrak{a}$, we have
\[
J(\Lambda)\left(\text{Exp}_{H}\cdot\Theta_{\phi}(T\mathcal{B})\right)\big\vert_{p_{w}}(\Lambda)=M^{*}(w,\Lambda)e^{\langle\rho+w\Lambda,H\rangle}
\]
\end{cor}

\subsection{Equivariant integral representation for $P_{\omega}(\nu,H)$}

For $\omega\in W/W_{\kappa}$ and $H\in\mathfrak{a}$, let 
\[
\mathcal{V}_{\omega}=T\mathcal{B}\big\vert_{\mathcal{C}_{\omega}},\quad\mathcal{E}_{\omega,H}=\text{Exp}_{H}\big\vert_{\mathcal{C}_{\omega}}
\]
Define 
\[
\Xi_{\omega,H}=e_{\dT}\left(T\mathcal{C}_{\omega}\right)\Theta_{\phi}(\mathcal{V}_{\omega})\mathcal{E}_{\omega,H}
\]
Recall that 
\[
P_{\omega}(\nu,H):=\left(\sum_{\sigma\in W_{\kappa}}\left(M^{*}(\omega\sigma,\rho_{\kappa}+\nu)e^{\langle\omega\sigma\nu,H\rangle}\right)\right)
\]
and
\[
\lim_{\nu\to0}P_{\omega}(\nu,H)=:P_{\omega}(H)
\]

\begin{prop}
\label{prop:equivariant-integral-rep}As meromorphic germs at $\rho_{\kappa}$,
we have the equivariant integral representation
\[
J(\Lambda)\left(\int_{\mathcal{C}_{\omega}}^{\dT}\Xi_{\omega,H}\right)(\Lambda)=P_{\omega}(\nu,H)e^{\langle\rho+\omega\rho_{\kappa},H\rangle},\quad\Lambda=\rho_{\kappa}+\nu
\]
\end{prop}

\begin{proof}
We apply the equivariant integration formula where the factor from
$e\left(T\mathcal{C}_{\omega}\right)$ in $\Xi_{\omega,H}$ cancels
the localization denominator: 
\[
J(\Lambda)\left(\int_{\mathcal{C}_{\omega}}^{\dT}\Xi_{\omega,H}\right)(\Lambda)=\left(\sum_{\sigma\in W_{\kappa}}J(\Lambda)\left(\text{Exp}_{H}\cdot\Theta_{\phi}(T\mathcal{B})\right)\big\vert_{p_{\omega\sigma}}(\Lambda)\right)
\]
\[
=\sum_{\sigma\in W_{\kappa}}M^{*}(\omega\sigma,\Lambda)e^{\langle\rho+\omega\sigma\Lambda,H\rangle}.
\]
\end{proof}

\subsection{Specialization-Evaluation at $\rho_{\kappa}$}

As before, $\wA:=\kappa(\mathbf{G}_{m})$. Since $\wA$ acts trivially
on $\mathcal{C}_{\omega}$, we have 
\[
H_{\wA}^{*}(\mathcal{C}_{\omega}):=H^{*}(\mathcal{C}_{\omega}\times_{\wA}E\wA)=H^{*}(\mathcal{C}_{\omega}\times B\wA)=H^{*}(\mathcal{C}_{\omega})\otimes_{\mathbb{C}}\mathbb{C}[z]
\]
Here $z$ denotes coordinates on $X_{*}(\wA)\otimes_{\mathbb{Z}}\mathbb{C}$
for which $z=1$ corresponds to $h_{\kappa}$ and $z=\frac{1}{2}$
corresponds to $\rho_{\kappa}=h_{\kappa}/2$. 

Let $\mathscr{O}_{\frac{1}{2}}$ be the ring of holomorphic germs
at $z=1/2$. By restriction (\cite{Fulton-Anderson}, p. 36), we obtain
\[
\text{Res}_{\wA}:H_{\dT,\text{an}}^{*}(\mathcal{C}_{\omega})_{\rho_{\kappa}}\to H^{*}(\mathcal{C}_{\omega})\otimes\mathscr{O}_{\frac{1}{2}}
\]
and the evaluation map 
\[
\text{eval}_{1/2}:H^{*}(\mathcal{C}_{\omega})\otimes\mathscr{O}_{\frac{1}{2}}\to H^{*}(\mathcal{C}_{\omega}),\quad\theta\otimes f\mapsto f(1/2)\theta
\]
Let 
\[
\text{sp}_{\kappa}=\text{eval}_{1/2}\circ\text{Res}_{\wA}
\]
Naturality of equivariant Gysin maps under change of groups (\cite{Fulton-Anderson},
p. 36) gives
\[
\left(\int_{\mathcal{C}_{\omega}}^{\dT}\eta\right)(\rho_{\kappa})=\int_{\mathcal{C}_{\omega}}\text{sp}_{\kappa}(\eta),\quad\eta\in H_{\dT,\text{an}}^{*}(\mathcal{C}_{\omega})_{\rho_{\kappa}}
\]
We note that the right hand side is an ordinary integral of an ordinary
cohomology class $\text{sp}_{\kappa}(\eta)\in H^{*}(\mathcal{C}_{\omega})$
over $\mathcal{C}_{\omega}$. 

\subsection{Conclusion of the proof}

The following is the special feature of a torus acting trivially on
the base: every equivariant Chern root is an ordinary Chern root plus
a scalar weight shift. If
\[
x_{1},\dots,x_{r}
\]
are the ordinary Chern roots of $\mathcal{V}_{\omega}[m]$, since
$\wA$ acts trivially on $\mathcal{C}_{\omega}$, the equivariant
Chern roots of $\mathcal{V}_{\omega}[m]$ are 
\[
x_{1}+mz,\dots,x_{r}+mz
\]
and after specialization at $\rho_{\kappa}$, they become 
\[
x_{1}+\frac{m}{2},\dots,x_{r}+\frac{m}{2},\quad z=\frac{1}{2}.
\]

\begin{prop}
\label{prop:specialization}There is a class $\Upsilon_{\omega,H}\in H^{*}(\mathcal{C}_{\omega})$
such that 
\[
\text{sp}_{\kappa}\left(\Xi_{\omega,H}\right)=e(E_{\omega})\Upsilon_{\omega,H}
\]
\end{prop}

\begin{proof}
This is a simple computation involving Chern roots. Decompose 
\[
\mathcal{V}_{\omega}=\bigoplus_{m\in\mathbb{Z}}\mathcal{V}_{\omega}[m]
\]
by $\wA$-weights. The computation is split into steps. 

\emph{The zero-weight tangent block.} If $x_{1}^{\wA},\dots,x_{\#}^{\wA}$
are the equivariant chern roots of $\mathcal{V}_{\omega}[0]=T\mathcal{C}_{\omega}$
and $\phi^{\circ}(s):=s\phi(s)$, then 
\[
e_{\wA}\left(T\mathcal{C}_{\omega}\right)\Theta_{\phi}(\mathcal{V}_{\omega}[0])=\prod_{j=1}^{\#}\phi^{\circ}(x_{j}^{\wA})
\]
Since $\phi^{\circ}$ is holomorphic at the origin, we have 
\[
\text{sp}_{\kappa}\left(e_{\wA}\left(T\mathcal{C}_{\omega}\right)\Theta_{\phi}(\mathcal{V}_{\omega}[0])\right)\in H^{*}(\mathcal{C}_{\omega})
\]

\emph{The weight $-2$ block.} If $x_{1}^{\wA},\dots,x_{\#}^{\wA}$
are the equivariant Chern roots of $E_{\omega}:=\mathcal{V}_{\omega}[-2]$,
then
\[
x_{j}^{\wA}=x_{j}-2z
\]
where $x_{1},\dots,x_{\#}$ are the ordinary Chern roots of $E_{\omega}$.
Let $\phi(s)=(s+1)\phi'(s)$, then
\[
\text{sp}_{\kappa}\left(\Theta_{\phi}(\mathcal{V}_{\omega}[-2])\right)=\prod_{j=1}^{\#}x_{j}\cdot\phi'(-1+x_{j})=e(E_{\omega})\cdot(\text{another class})
\]
since $\phi'$ is holomorphic at $-1$. 

Since $\phi$ is holomorphic on $\mathbb{R}-\{0\}$ and the exponential
factor is an analytic class, we have 
\[
\text{sp}_{\kappa}\left(\Xi_{\omega,H}\right)=e(E_{\omega})\Upsilon_{\omega,H}
\]
for some $\Upsilon_{\omega,H}\in H^{*}(\mathcal{C}_{\omega})$. 
\end{proof}
We now finish the proof of the main Theorem \ref{thm:main}. 
\begin{thm*}
Let $\kappa:\slC\to\dG$ be a distinguished homomorphism with $h_{\kappa}=\kappa\vert_{\mathbf{G}_{m}}$
dominant. Let $\rho_{\kappa}=h_{\kappa}/2\in\da$. The non-zero automorphic
form $E^{*}(\rho_{\kappa})$ is square-integrable. 
\end{thm*}
\begin{proof}
The constant term of $E^{*}(\rho_{\kappa})$ along $B$ is
\[
CE^{*}(\rho_{\kappa};g)=\sum_{\omega\in W/W_{\kappa}}P_{\omega}(H)e^{\langle\rho+\omega\rho_{\kappa},H\rangle},\quad H:=H(g)
\]
From Proposition \ref{prop:specialization}, we have 
\[
P_{\omega}(H)e^{\langle\rho+\omega\rho_{\kappa},H\rangle}=J(\rho_{\kappa})\int_{\mathcal{C}_{\omega}}e(E_{\omega})\Upsilon_{\omega,H}
\]
Theorem \ref{thm:euler-vanish} gives
\[
\omega\rho_{\kappa}\notin-\mathscr{C}_{\text{rt}}\quad\implies\quad e(E_{\omega})=0
\]
and hence 
\[
P_{\omega}=0\quad\text{if}\quad\omega\rho_{\kappa}\notin-\mathscr{C}_{\text{rt}}.
\]
By Langlands' criterion (Proposition \ref{prop:Langlands-criterion})
\[
E^{*}(\rho_{\kappa})\in L^{2}[G].
\]
\end{proof}
\printbibliography

\end{document}